\documentclass[12pt,reqno]{amsart}
\usepackage {amssymb}
\usepackage {amsmath}
\usepackage {bbm}
\usepackage{amsthm}
\usepackage{mathtools}
\usepackage{graphicx}
\usepackage {amscd}
\usepackage[alphabetic]{amsrefs}
\usepackage[colorlinks, citecolor=blue]{hyperref}
\usepackage{enumerate}
\usepackage{tikz}
\usepackage{tikz-cd}
\usepackage{verbatim,color,geometry}
\usepackage[all]{xy}
\usepackage{enumitem}
\usepackage{chngcntr}
\counterwithout{equation}{section}
    \newenvironment{dedication}
        {\vspace{6ex}\begin{quotation}\begin{center}\begin{em}}
        {\par\end{em}\end{center}\end{quotation}}

\newtheorem{prop}{Proposition}[section]
\newtheorem{thm}[prop]{Theorem}
\newtheorem{cor}[prop]{Corollary}
\newtheorem{conj}[prop]{Conjecture}
\newtheorem{lem}[prop]{Lemma}

\theoremstyle{definition}

\newtheorem{defn}[prop]{Definition}
\newtheorem{exmp}[prop]{Example}
\newtheorem{rem}[prop]{\it Remark}

\newtheorem{say}[prop]{}
\newtheorem{lem-defn}[prop]{Lemma-Definition}
\newtheorem{defn-lem}[prop]{Definition-Lemma}
\newtheorem{claim}[prop]{Claim}
\newtheorem{construction}[prop]{Construction}

\newtheorem*{claim*}{Claim}

\newcommand{\bR}{\mathbb{R}}
\newcommand{\bA}{\mathbb{A}}
\newcommand{\bQ}{\mathbb{Q}}
\newcommand{\bZ}{\mathbb{Z}}
\newcommand{\bN}{\mathbb{N}}

\newcommand{\bG}{\mathbb{G}}

\newcommand{\bin}{{\rm {\bf in}}}

\newcommand{\hvol}{\widehat{\rm vol}}

\newcommand{\cX}{\mathcal{X}}
\newcommand{\cY}{\mathcal{Y}}

\newcommand{\cO}{\mathcal{O}}

\newcommand{\cR}{\mathcal{R}}

\newcommand{\cE}{\mathcal{E}}

\newcommand{\cP}{\mathcal{P}}

\newcommand{\fa}{\mathfrak{a}}
\newcommand{\fb}{\mathfrak{b}}
\newcommand{\fc}{\mathfrak{c}}
\newcommand{\fm}{\mathfrak{m}}

\newcommand{\Spec}{\mathrm{Spec}~}

\newcommand{\mult}{\mathrm{mult}}
\newcommand{\lct}{\mathrm{lct}}

\newcommand{\ord}{\mathrm{ord}}

\newcommand{\wt}{\mathrm{wt}}
\newcommand{\Val}{\mathrm{Val}}

\newcommand{\gr}{\mathrm{gr}}
\newcommand{\ui}{\underline{i}}

\newcommand{\DR}{\mathcal{DR}}

\newcommand{\LC}{{\rm LC}}

\numberwithin{equation}{section}

\begin{document}

\title[Toward higher rank finite generation]{Toward finite generation of higher rational rank valuations}

\begin{dedication}
\hspace{0cm}
\vspace*{1cm}{Dedicated to Vyacheslav Shokurov's Seventieth Birthday}
\end{dedication}

\author{Chenyang Xu}
\address{Current address: Department of Mathematics, Princeton University, Princeton, NJ, 08544.}
\email{chenyang@math.princeton.edu}
\address{Department of Mathematics, MIT, Cambridge, MA, 02139.}
\email{cyxu@math.mit.edu}
\address{BICMR, Peking University, Beijing, 100871.}
\email{cyxu@math.pku.edu.cn}

\date{}

\begin{abstract}{We propose a finite generation conjecture for the valuation which computes the stability threshold of a log Fano pair. We also initiate a degeneration strategy of attacking the conjecture.  
}\end{abstract}

\maketitle

\setcounter{tocdepth}{1}
\tableofcontents

\section{Introduction}
The finite generation of the ring given by the sections of all multiples of a fixed divisor is a central question  in algebraic geometry. While the answer is not always affirmative for any arbitrary divisor, the minimal model program provides a very powerful tool to establish the finite generation when the divisor is given by the log canonical class, see e.g. \cite{BCHM}.

In \cite{Li-minimizer, LX-SDCII}, we posted a series of conjectures, called \emph{the Stable Degeneration Conjecture}, which attaches to \emph{any} Kawamata log terminal (klt) singularity a unique K-semistable log Fano cone degeneration, through the minimizing valuation of the normalised volume function defined in \cite{Li-minimizer}.  There have been many works toward this conjecture, especially it is known that any minimizer is quasi-monomial and an lc place of a $\bQ$-complement by \cite{Xu-quasimonomial}. The only remaining part is the following.

\begin{conj}[{\cite{Li-minimizer}}]\label{c-weak}
Let $x\in (X={\rm Spec} \ R,\Delta)$ be a germ of klt singularity. Let $v$ be a (quasi-monomial) minimizer of $\hvol_{X,\Delta}$, then the associated graded ring ${\rm gr}_v R$ is finitely generated. 
\end{conj}

This is known in dimension 2 (see \cite{Cutkosky-twodimensional}*{Proposition 1.2}), but not in general.
We can also similarly ask a global variant of Conjecture \ref{c-weak}.
\begin{conj}\label{c-main}
Let $(X,\Delta)$ be a log Fano variety with $\delta(X,\Delta)\le 1$ and $r$ a positive integer such that $-r(K_X+\Delta)$ is Cartier. Let $R=\bigoplus_{m\in \bN}H^0(-mr(K_X+\Delta))$. If  a valuation $v$ computes $\delta(X,\Delta)$, then ${\rm gr}_vR$ is finitely generated. 
\end{conj}
In \cite{BLX-openness}, we have shown under this assumption,  any valuation computing $\delta(X,\Delta)$ is quasi-monomial, and it is always a lc place of a $\bQ$-complement of $(X,\Delta)$.

\medskip

In this note, we discuss these conjectures and establish a degeneration process to investigate the question. In fact, since both these two kinds of minimizers in Conjecture \ref{c-weak} and \ref{c-main} are lc places of a complement, we first study a general degeneration process for a set of divisorial valuations which are lc places of the same $\bQ$-complement. We only state the local one as follows, as the global one can be considered as a special case of the local one via the cone construction.
\begin{thm}\label{t-degeneration}
Let $x\in (X={\rm Spec} (R),\Delta)$ be a klt pair and $\fa$ an $\fm_x$-primary ideal with $c=\lct(X,\Delta;\fa)$. Then for any set of prime divisors $E_1$,..., $E_r$ which are lc places of $(X,\Delta+\fa^c)$, there is a $\bG_m^r$-equivariant locally stable family $(\cX,\Delta_{\cX})\to \bA^r$, such that if we restrict on the $r$ axes, we obtain the induced degenerations of $X$ to $\Spec({\rm gr}_{E_i} R)$ $(i=1,...,r)$ given by the extended Rees algebra.
\end{thm}

In Section \ref{ss-degeneration}, we will provide  more detailed information for our construction. 

In  \cite[Problem 3.1]{AIM-conjecture}, a stronger conjecture was posted to predict the finite generation of $\gr_v(R)$ for any lc place $v$ of $(X,\Delta+\fa^c)$. However, a counterexample was recently found (see \cite[Theorem 1.4]{AZ-adjunction}).   Nevertheless, we establish the following criterion.

\begin{thm}\label{t-main2}
Notation as above. For a rational rank $r$  valuation $v$ which is an lc place of $(X,\Delta+\fa^c)$, the associated graded ring $\gr_v R$ is finitely generated if and only if we can find $r$ prime divisors $E_1,..., E_r$ which are also lc places of $(X,\Delta+\fa^c)$, such that
\begin{enumerate}
\item there is a model $Y\to X$ extracting $E_1,..., E_r$, and $v$ is a toroidal divisor over a point $\eta\in (Y, E_1+\cdots +E_r)$;
\item the degeneration constructed in Theorem \ref{t-degeneration} for $E_1,...,E_r$ has an irreducible fiber over $0\in \bA^r$.
\end{enumerate}
\end{thm}

\begin{rem}To proceed, it will be crucial to further pin down a birational geometry condition to guarantee the irreducibility assumption in Theorem \ref{t-main2} and cover the K-stability type assumption in Conjecture \ref{c-weak} and \ref{c-main}.
\end{rem}
\medskip

\noindent {\bf Acknowledgement:} We would like to thank Harold Blum, Yuchen Liu, Mircea Musta\c{t}\v{a} and Ziquan Zhuang for helpful discussions. We are also grateful to the anonymous referee for many useful suggestions.

\section{Preliminaries}

\noindent{\bf Notation and Conventions:} We follow the terminology in \cite{KollarMori, Kollar13, Laz-positivityII}. See \cite[Def. 2.34]{KollarMori} for the definitions of lc, klt and dlt singularities for a log pair $(X,\Delta)$. Also see \cite[Def. 35]{dFKX-dualcomplex} for the definition qdlt singularities. See \cite[Sect. 5.2]{Kollar13} for the definition of {\it semi-log canonical (slc)} singularities. 

Let $X$ be a reduced, irreducible (separated) variety defined over $k$. Denote by $\Val_X$ 
the set of real valuations of $K(X)$ that admit a center
on $X$. For a closed point $x\in X$, we denote by $\Val_{X,x}$ the set
of real valuations of $K(X)$ centered at $x\in X$. 

Let $X$ be a normal variety, we say that $X$ is {\it potentially klt} if there is an effective $\mathbb Q$-divisor $\Delta$ such that $(X,\Delta)$ is klt.

\subsection{Grading and $T$-action}



 


\subsubsection{Graded sequence of ideals and valuations}

Let $R$ be a ring. Let $\Phi$ be a monoid, and we call $\{\fa_{i}\}_{i\in \Phi}\subset R$ {\it a graded sequence of ideals}, if $\fa_i\cdot \fa_j\subset \fa_{i+j}$ for any $i,j\in \Phi$. We always assume $\fa_0=R$.  Moreover, if the monoid $\Phi$ is partially ordered, then we always assume $\fa_i\supset \fa_j$ for any given pair $i\le j$.

If there is an embedding $\Phi\subset \bR_{\ge 0}$ as monoids, for any $i\in \Phi$, we define $\fa_{>i}=\cup_{j>i}\fa_j$. We can form {\it the associated graded ring} 
$${\rm gr}(R,\fa_{\bullet}):= \bigoplus_{i\in \Phi}\fa_i/\fa_{>i}.$$ 

Assume $R$ is an integral domain and denote by $X=\Spec R$,  we let $\Val_X$ be the space of all real valuations of $K(X)$ whose center is on $X$. Assume $v\in  \Val_{X}$, let $\Phi$ be the value monoid of $v$, i.e. $\Phi=\{v(f) |~ f\in R\}$. For any $t\in \bR$, we can consider the (possibly trivial) ideal $\fa^v_{t}:=\{f\in R | ~ v(f)\ge t\}$.  

The following easy lemma says when a graded sequence of ideals arises from a valuation. 

\begin{lem}\label{l-filtrationrestriction}
When $R$ is an integral domain, and $\fa_{\bullet}$ is graded by a monoid in $\bR_{\ge 0}$, then ${\rm gr}(R,\fa_{\bullet})$ is an integral domain if and only if $\fa_{\bullet}$ is the graded sequence of valuative ideals for some real valuation $v$. 
\end{lem}
\begin{proof}See \cite[Page 8]{Tei-valuation}.
\end{proof}

For any $r$-tuples in $\bR^r$, we say $\ui=(i_1,...,i_r)\ge \underline{j}=(j_1,...,j_r)$ if $i_k\ge j_k$ for any $1\le k \le r$ and $(i_1,...,i_r)>(j_1,...,j_r)$ if $i_k\ge j_k$ for any $k$ and at least one inequality is strict. For a graded  sequence of ideals $(\fa_{\ui})_{\ui\in \bZ^r_{\ge 0}}$ graded by the monoid $\mathbb{Z}^r_{\ge 0}$,  we can define the associated graded ring 
$${\rm gr}(R,\fa_{\bullet}):=\bigoplus_{\ui\in \bZ^r_{\ge 0}}\fa_{\ui}/\fa_{>\ui}.$$

\subsubsection{Extended Rees algebra}

For a filtration $\{\fa_{\bullet}\}_{\bZ^r_{\ge 0}}$, we can construct an extended Rees $k[t_1,...,t_r]$-algebra $\cR$.  We will be mainly interested in the case that the associated graded ring ${\rm gr}(R,\fa_{\bullet})$ is finitely generated. The grading on ${\rm gr}(R,\fa_{\bullet})$ automatically gives a $T=(\bG_m)^r$-action on it. 

\begin{exmp}[Extended Rees algebra in rank one]\label{e-era}Let $(\fa_{i})_{i\in \bZ_{\ge 0}}$ be a graded sequence of ideals of a finitely generated $k$-algebra $R$ with $\fa_0=R$. Then $\bigoplus_{i\in \bZ_{\ge 0}}~\fa_i$ is finitely generated if and only if the associated graded ring
$$\bigoplus_{i\in \bZ_{\ge 0}}\fa_i/\fa_{i+1} \mbox{ is finitely generated}.$$
In this case, we can form the following {\it extended Rees ($k[t]$-)algebra} $\cR:=\bigoplus_{i\in \bZ} \fa_it^{-i}$ where we assume $\fa_i=R$ for $i\le 0$. Then the $k[t]$-algebra structure together with the grading means there is a  $\bG_m$-equivariant morphism $\Spec\cR\to \bA^1$, where 
$$\Spec\cR\times_{\bA^1}(\bA^1\setminus \{0\})=(X:={\rm Spec}(R))\times (\bA^1\setminus\{0\})$$ and the special fiber over $0$ is $\Spec \bigoplus_{i\in \bZ_{\ge 0}}\fa_i/\fa_{i+1} $.
\end{exmp}

The $T$-families in our note come from the following construction, which is a generalization of Example \ref{e-era} from a point to a more general base.

\begin{construction}\label{extendedrees}
Let $(X_B={\rm Spec}~(R))\to (B={\rm Spec}~(A))$ be a dominating morphism between two affine varieties. Let $\{\fa_{\bullet}\}_{k\in  \bN}$ be a graded sequence of ideals of $R$, with $\fa_0=R$ and denote by $\fa_i=R$ for $i\le 0$. Assume  $\fa_i\cap A=0$ for any $i>0$.  

Assume  $\bigoplus_i \fa_i$ is finitely generated over $A$. Then we can consider {\it the relative extended Rees ($k[t]$-)algebra}
\begin{center}
\begin{tikzcd}[column sep=scriptsize, row sep=scriptsize]
 A[t]=\bigoplus_{i\le 0} A\cdot t^{-i}  \arrow[d] & \subset  &  A\otimes_k k[t,\frac{1}{t}] \arrow[d]\\
\cR:=  \bigoplus_{i\in \bZ}\fa_i\cdot t^{-i} &\subset  & R\otimes_k k[t,\frac{1}{t}]
   \end{tikzcd}
   \end{center}
Geometrically, we have $\cX:=\Spec   \cR \to  \bA^1_A.$

Now we assume $R/\fa_i$ is \emph{flat} over $A$ for any $i$, then the above construction commutes with base change in the following sense: For 
$$B'=\Spec A'\to B \mbox{ given by \ \ }  A\to A',$$ let $\fb_i:=\fa_i\otimes_A  A'$. Then since $R/\fa_i$ is flat, $(\fb_i)_{i\in \bZ_{\ge 0}}$ is a graded sequence of ideals satisfying the conditions at the beginning.  
Moreover, if $\bigoplus_i \fa_i$ is finitely generated over $A$, then  $\bigoplus_i \fb_i$ is finitely generated over $A'$. Base changing over $B'$, we get 
$$\bigoplus_{i\in \bZ} \fb_i\cdot t^{-i}\subset R\otimes_A A'\otimes k[t,\frac{1}{t}],$$
and geometrically, we have a base change morphism
$$\cX\times_B B'=\Spec \bigoplus_{i\in \bZ} \fb_i\cdot t^{-i}\to \bA^1_{A'}.$$ 
\end{construction}

\begin{lem}\label{l-initial}With the same notation as in Construction \ref{extendedrees}, 
assume $R$ is flat over $A$. Then the above family $\Spec (\cR)\to \bA_A^1$ is flat.

\end{lem}
\begin{proof} For any $k\in \mathbb N$, let $\cR_k=\bigoplus_{j\le k}\fa_jt^{-j}$. It suffices to show that $\cR_k$ is a flat $A[t]$-module for any $k$, as $ \cR=\bigcup_k \cR_k$. 
Define the ideal $\fc_i$ of $\cR$ by
$$\fc_i=(\bigoplus^{\infty}_{j=i+1}\fa_{j}t^{-j}) \bigoplus (\bigoplus^{i}_{j=-\infty}\fa_{i}\cdot t^{-j})$$  and    $\fc_{i,k}=\fc_i\cap \cR_k$.
For $0\le i\le k-1$, there is an exact sequence of $A[t]$-modules
$$0\to \bigoplus^{-\infty}_{j= i}(\fa_{i}/\fa_{i+1})t^{-j}\to \cR_k/\fc_{i+1,k}\to \cR_k/\fc_{i,k}\to 0.$$
As an $A[t]$-module, $\bigoplus^{-\infty}_{j= i}(\fa_{i}/\fa_{i+1})t^{-j}$ is isomorphic to $(\fa_{i}/\fa_{i+1})[t]$ which is flat over $A[t]$ as $\fa_{i}/\fa_{i+1}$ is flat over $A$.
By induction for any $i\ge 0$,  $\cR_k/\fc_{i,k}$ is flat over $A[t]$ as $\fc_{0,k}=\cR_k$. 
Since $R$ and $R/\fa_{i}$ are flat over $A$, we know $\fa_i$ is flat over $A$. Since $\fc_{k,k}\cong \fa_k[t]$ as $A[t]$-modules, $\fc_{k,k}$ is flat over $A[t]$, which implies that $\cR$ is flat over $A[t]$.
\end{proof}

\begin{defn}\label{d-degI}In the above setting, let $I\subset R$ be an ideal, then we define the ideal 
$$\tilde{I}:=\bigoplus_{i\in\bZ}(I\cap \fa_i)~t^{-i}\subset \cR$$ to be the flat family of ideals which degenerates $I$, i.e., the vanishing locus of $V(\tilde{I})$ is the closure of $V(I)\times (\bA^{1}\setminus\{0\})$ in $\cX$. 

In particular, for a family of $\bQ$-divisors $\Delta_B$ on $X_B$ over $B$ which can be written as $\sum_i d_i\Delta_i$ where $\Delta_i$ is an irreducible divisor on $X_B$ given by an ideal $I_i$, then for each $i$, $\tilde{I}_i$ given above defines the degeneration family $\widetilde{\Delta}_i$ of $\Delta_i$ over $B\times \bA^1$. We define $\Delta_{\cX}=\sum_i d_i\widetilde{\Delta}_i$.
\end{defn}

\begin{lem}\label{l-extendprimary}
Let $x\in X=\Spec {R}$, and  a graded sequence $\{\fa_i\}_{i\in \bZ_{\ge 0}}$ of $\fm_x$-primary ideals. Assume $\bigoplus_{i\in \bZ_{\ge 0}}\fa_i$ is finitely generated and denote by $\cR=\bigoplus_{i\in \bZ}\fa_it^{-i}$. Let $\fb$ be an $\fm_x$-primary ideal. Then $\bigoplus_{i\in \bZ} (\fb\cap\fa_i)t^{-i}$ satisfies that ${\rm Cosupp}(\bigoplus_{i\in \bZ} (\fb\cap\fa_i)t^{-i})$ is the unique flat extension over $\mathbb{A}^1$ of ${\rm Cosupp}(\fb)\times ( \bA^1\setminus\{0\})$ on $\cX:=\Spec\cR$.
\end{lem}
\begin{proof}
It is easy to check $\bigoplus_{i\in \bZ} (\fb\cap\fa_i)t^{-i}$ satisfies that ${\rm Cosupp}(\bigoplus_{i\in \bZ} (\fb\cap\fa_i)t^{-i})$ is flat (see e.g. \cite[Lem 4.1]{LX-SDCI}). 
Such an extension is unique because the Hilbert scheme (of length $\dim_k(R/\fb)$) is proper over $\mathbb{A}^1$. Thus we get the above statement.
\end{proof}

\begin{exmp}\label{em-ambient}
Consider the $(\mathbb{G}_m)^r$-action on $\bA^N$ given by the following:
$$(\mathbb{G}_m)^r=\bG_m\times \cdots \bG_m, \mbox{ and the $k$-th $\bG_m$-component acts by the weight }(w_{k,1},..., w_{k,N})\in \bN^r.$$
Denote by $P:=k[x_1,....,x_N]$. Then define the ring 
$$\cP:=\bigoplus_{\ui\in \mathbb{Z}^N}F^{\ui}P\cdot  t_1^{-i_1}\cdots t_r^{-i_r}\subset P\otimes k[t_1,t_1^{-1},...,t_r,t_r^{-1}],$$
where $F^{\ui}P\subset P$ is the ideal generated by the monomials $x_1^{d_1}\cdots x_N^{d_N}$ such that 
\begin{equation}\label{e-weight}
\sum^N_{m=1}d_{m}\cdot w_{k,m}\ge i_k \mbox{\ \  for any $1\le k\le r$}.
\end{equation}
We have a $k[t_1,...,t_r ]$-algebra
$$\cP=\bigoplus_{(d_1,...,d_N)\in \bZ^{N}_{\ge 0}}x_1^{d_1}\cdots x_N^{d_N}t_1^{-\sum^N_{m=1}d_{m}\cdot w_{1,m}}\cdots t_r^{-\sum^N_{m=1}d_{m}\cdot w_{r,m}}.$$
There is an isomorphism $\cP\cong k[y_1,...,y_N,t_1,...,t_r]$ by sending 
$$y_m\to x_m\cdot t_1^{-w_{1,m}}\cdots t_r^{-w_{r,m}}.$$ 

Then $k[t_1,...,t_r]\subset \cP$ gives  a $\bG_m^r$-equivariant family $\bA^N\times \bA^r\to \bA^r$.

This could be considered as repeatedly using Constuction \ref{extendedrees} $r$-times, each time for the filtration given by \eqref{e-weight}, i.e., induced by the action of $k$-th component in $\bG^r_m$  for $1\le k \le r$.

\end{exmp}

\subsection{Log pairs}
\subsubsection{A log pair with an ideal}
We extend some of the previous discussions in the setting of a log pair with an ideal. 

\begin{defn}[Log canonical thresholds]Let $(X,\Delta)$ be a log canonical pair and $\fa$ an ideal on $(X,\Delta)$. We  define the log canonical threshold at a point $x$ on ${\rm Cosupp}(\fa)$ to be
$$\lct_x(X,\Delta,\fa)=\min_w\frac{A_{X,\Delta}(w)}{w(\fa)},$$  
where $w$ runs over all valuations with $x\in {\rm Cent}_X(w) \subset {\rm Cosupp}(\fa)$.
If $x$ is not on ${\rm Cosupp}(\fa)$, we set $\lct_x(X,\Delta,\fa)=+\infty$. And 
$$\lct(X,\Delta;\fa)=\min_{x\in X}\lct_x(X,\Delta,\fa).$$ 

More generally, if $(X,\Delta)$ is a semi-log canonical pair, $\fa$ is an ideal on $(X,\Delta)$, and $\pi\colon X^{ n}\to X$ is the normalization, and we write 
$$\pi^*(K_X+\Delta)=K_{X^{ n}}+\Delta^{ n}. $$
Then $(X^n,\Delta^n)$ is log canonical, and we define 
$$ \lct_x(X,\Delta,\fa)=\min_{\pi(y_i)=x}\lct_{y_i}(X^n,\Delta^n,\fa\cdot \mathcal{O}_{X^n}). $$
\end{defn}

\begin{say}\label{r-ideal}
Let $(X,\Delta)$ be an lc pair, $c$ a non-negative rational number and $\fa$ an ideal on $X$.  Fix any point $p\in X$.  Then by restricting over an affine neighborhood $U$ of $p$, $c\cdot \mult_{F}(\fa)\le A_{X,\Delta}(F)$ for all divisors $F$ whose center contains $p$ if and only if $(X,\Delta+\frac{c}{k}(H_1+\cdots+H_k))$ is lc around $p$ where  $H_1,...,H_k$ are general hypersurfaces vanishing along $\fa$ and $k\ge c$.  This can be seen as follows:  Let $\mu\colon Y\to X$ be the normalised blow up along $\fa$, then $\mu^{-1}\fa=\mathcal{O}_Y(-E)$ is ample and base point free over $U$. Thus $\mu^*(H_i)(-E)$ are contained in a base point free linear system. In this case, we will say {\it $(X,\Delta+\fa^c)$ is lc} around $p$.

We can similarly define {\it $(X,\Delta+\fa^c)$ slc} by requiring $(X,\Delta)$ is slc and $\big(X^n,\Delta^n+(\fa\cdot \mathcal O_{X^n})^c\big)$ is lc. 
\end{say}

\begin{lem}\label{l-dualcomplex}
Let $(X,\Delta)$ be a klt pair, $\fa$ an ideal on $(X,\Delta)$ with $c=\lct(X,\Delta;\fa)$. Let $\mu\colon (Y,E)\to (X,\Delta+\fa)$ be a log resolution. 

Then all valuations that satisfy $A_{X,\Delta}(v)=c\cdot v(\fa)$ are precisely the cone over the dual complexes formed by the prime components of $E$ which compute the log canonical threshold $\lct(X,\Delta;\fa)$.
\end{lem}
\begin{proof}See \cite[Lem. 2.4]{Xu-quasimonomial}.
\end{proof}

\subsubsection{Adjunction}

Let $(X,\Delta)$ be an slc pair, and $E\subset \lfloor \Delta \rfloor$ a reduced (possibly reducible) divisor.  We can define (see \cite[Definition 4.2]{Kollar13})
$$(K_X+\Delta)|_{E}=K_{E}+\Delta_{E}, $$
where $\Delta_E:={\rm Diff}_E(\Delta)$. If $E$ is $S_2$, then $E$ is demi-normal (see \cite{Kollar13}*{Definition 5.1}), and $(E,\Delta_E)$ is slc. 

For an ideal $\fa$, we can define $\fa|_{E}$ as the image of $\fa\to \cO_X\to \cO_E$.

\begin{prop}[Inversion of adjunction]\label{p-IOA}
Let $(X,\Delta)$ be an slc pair, and $E\subset \lfloor \Delta \rfloor$ and an ideal $\fa$ on $X$. Assume $E$ is $S_2$.
 Then $(X,\Delta+\fa^c)$ is slc along a neighborhood of $E$ if and only if $(E,\Delta_{E}+(\fa|_E)^c)$ is slc. 
\end{prop}
\begin{proof}
This is standard if $(X,\Delta)$ is log canonical (see \cite[Theorem 4.9]{Kollar13}). For slc version, we can easily reduce to the lc version by the following. 

Denote by $\pi\colon X^n\to X $ the normalization with the conductor divisor $D\subset X^n$, and $\pi_E\colon E^n\to E$ the normalization of $E$ with $i:E\hookrightarrow X$ and $i^n: E^n\to X^n$. We have $(i^{-1}\fa)|_{E^n}=(i^n)^{-1}(\fa\cdot \mathcal{O}_{X^n})$. If  $(E,\Delta_{E}+(\fa|_E)^c)$ is slc, then $(E^n,\Delta_{E^n}+i^{-1}(\fa|_E)^c)$ is lc, which implies $(X^n, \pi^*\Delta+D+(\fa\cdot \mathcal{O}_{X^n})^c)$ is lc by the above mentioned standard inversion of adjunction and Paragraph \ref{r-ideal}. Thus $(X,\Delta+\fa^c)$ is slc. 
\end{proof}

\begin{lem}\label{l-extract}
Let $(X,\Delta)$ be a quasi-projective klt pair. Assume a prime divisor $E$ computes the log canonical threshold $c:=\lct(X,\Delta;\fa)$ for some ideal $\fa$, then there exists a projective morphism $\mu\colon Y\to X$ such that ${\rm Ex}(\mu)=E$, $-E$ is ample over $X$ and $(Y,E+\mu_*^{-1}(\Delta))$ is log canonical. Moreover, for any $m$ such that $mE$ is Cartier, if we write $\mu^{-1}(\fa^m)=\cO_Y(-m\cdot\ord_{E}(\fa)E)\cdot \fb_m$, then $(Y,E+\mu_*^{-1}(\Delta)+\fb_m^{c/m})$ is log canonical. 
\end{lem}
\begin{proof}The existence of a model $Y$ which only extracts $E$ with $-E$ ample over $X$ is implied by \cite[1.3.1]{BCHM}. The rest of the claim follows from the definition. 
\end{proof}
In the above lemma, if $m_1|m_2$ and $m_1E$ is Cartier, then we have $\fb_{m_2}=\fb^{m_2/m_1}_{m_1}$. Thus by abuse of notation, we will write 
\begin{eqnarray}\label{e-Qideal}
\fb:=\fb_{m}^{1/m} \mbox{ for sufficiently divisible $m$ and } \mu^{-1}(\fa)=\cO_Y(-\ord_{E}(\fa)E)\cdot \fb
\end{eqnarray}
then $(Y,E+\mu_*^{-1}(\Delta)+\fb^{c})$ is log canonical.

\subsubsection{Locally stable family}

The definition of log pair for a family of pairs $(X,\Delta)\to B$ is subtle, especially when $B$ has bad singularities. See \cite{Kollar-modulibook} for more background. 
\begin{defn}[Locally stable family of pairs]\label{d-lsf}
A log pair $(X,\Delta)\to B$ over a normal base $B$ is called {\it locally stable}, if 
\begin{enumerate}
\item $X$ is flat over $B$, $\omega^{[m]}_{X/B}(m\Delta)$ is Cartier for some $m\in \bN$,
\item for any $s\in B$, $\Delta$ does not contain any component of $X_s$ or codimension one component of ${\rm Sing}(X_s)$, and
\item for any $s$, $(X_s,\Delta_s)$ is slc, where $\Delta_s$ is defined to be divisorial pull back of $\Delta$ on $X_s$.
\end{enumerate}
\end{defn}
The main properties we need for locally stable family is that for a base change $B'\to B$, we can define the pull back family. 
See \cite[Chapter 4]{Kollar-modulibook} for a comprehensive treatment.

We need the following easy lemma.
\begin{lem}\label{l-localstable}
Let $(X,\Delta)$ be a log canonical pair, and $\pi\colon X\to B$ a flat morphism to a smooth variety. Let $s\in B$ be a codimension $e$ point and $H_1$,..., $H_e$ general hypersurfaces, whose equations generate $\fm_s/\fm_s^2$. Assume $(X,\Delta+\pi^*(\sum^e_{i=1}H_i))$ is log canonical along $ \pi^{-1}(s)$, then $(X,\Delta)\to B$ is a locally stable family over a neighborhood of $s$.
\end{lem}
\begin{proof}By restricting over $\cap^{e-1}_{i=1}H_i$ and localization and applying induction,  we can assume $B$ is a curve, and $H_e=\{s\}$. Since any point on $X_s$ is not a log canonical center, we know $(X,\Delta)$ is $S_3$ along $X_s$ (see \cite[Cor. 7.21]{Kollar13}). Thus $X_s$ is $S_2$. 

Since $(X,\Delta+X_s)$ is log canonical, then by the classification of surface singularities,  ${\rm Supp}(\Delta)$ does not contain any component of $X_s$ as well as any codimension one singularity of $X_s$.

Thus we know $(X_s,\Delta_s)$ is slc. 
\end{proof}

\section{Degeneration to a locally stable family}

\subsection{Constructing locally stable degeneration}

\begin{lem}\label{l-kvvanishing}
Let $(X,\Delta)$ be a locally stable family over a smooth base $B$ with general fibers being  klt. Let $\sigma\colon B\to X$ be a section. Let $\mu\colon Y\to (X,\Delta)$ be a projective birational morphism with an irreducible exceptional divisor $E$ such that $\pi_Y\colon (Y,\mu_*^{-1}\Delta+E)\to B$ is a locally stable family, ${\rm Center}_X(E)=\sigma(B)$, and 
$$-E\sim_{X,\mathbb{Q}}-a(K_Y+\mu_*^{-1}\Delta+E)\ \ \mbox{where } a=A_{X,\Delta}(E)>0$$ is relatively ample over $X$. Then for any $m\ge 0$,  the quotient sheaf
$$\mu_*\cO_Y(-mE)/\mu_*\cO_Y(-(m+1)E)$$ is flat over $B$, and the morphism
$$ \mu_*\cO_Y(-mE)\otimes k_s \to \mu_{s*}\cO_{Y_s}(-mE_s)$$
is an isomorphism for any $s\in B$ and $m\in \bN$. In particular, $\mathcal{O}_X/\mu_*\cO_Y(-mE)$ is flat over $B$.
\end{lem}

\begin{proof}
Since $Y$ is a potentially klt space, and $E$ is $\mathbb{Q}$-Cartier, we know that $\mathcal{O}_Y(-mE)$ and $\cO_E$ is CM (see \cite[Cor. 5.25]{KollarMori}). Since  $(Y,E+\mu_*^{-1}\Delta)\to B$ is a locally stable family, $\pi_Y|_E\colon E\to B$ is equidimensional. Thus it is flat, as $B$ is smooth. 

Since 
$$-mE\sim_{\mathbb{Q}}K_Y+\mu_*^{-1}\Delta+E+(-am-1)(K_Y+\mu_*^{-1}\Delta+E),$$
and $(Y,\mu_*^{-1}\Delta+(1-\epsilon)E)$ is klt for any $\epsilon\in (0,1]$, by Kawamata-Viehweg vanishing theorem, for any positive integer $m\ge 0$
$$R^i\mu_*(\cO_Y(-mE))=0 \mbox {\ \ for any }i>0.$$

Assume $p$ to be a positive integer, such that $pE$ is Cartier on $Y$.
Define $Q_m$ by the following exact sequence 
\begin{equation}\label{e-long}
0\to \cO_Y((-m-1)E)\to \cO_Y(-mE)\to Q_m\to 0,
\end{equation}
then $Q_{m+p}=Q_m \otimes \cO_Y(-pE)$. 

Moreover, we claim $Q_m$ is flat over $B$ for any $m$.  This follows from the proof of \cite[Cor. 5.25]{KollarMori}. 
In fact, locally around any point $p\in E$, we can take a degree $p$ finite covering $Y'\to Y$ such that the pull back $E'$ on $Y'$ is Cartier. Then  the pushforward of $\cO_{E'}$ on $E$ is the direct sum $\bigoplus^{p-1}_{m=0}Q_m$. Since $E'$ is flat over $B$, the same is true for $Q_m$.

Since $E\to B$ is proper, $Q_m$ supports on $E$, thus $\mu_*Q_m$ is a finite $\cO_B$-module. The long exact sequence \eqref{e-long} implies that $R^i\mu_*(Q_m)=0$ for any $i>0$, which implies 
$\mu_*(Q_m)$ is flat over $B$. This implies that for any $m$, $\mu_*\cO_Y(-mE)$ is flat over $B$ and satisfies the base change, and $\mu_*\cO_Y(-mE)/\mu_*\cO_Y(-(m+1)E)\cong \mu_*Q_m$ is also flat over $B$.

Then the last statement follows from the fact that $\mathcal{O}_X/\mu_*\cO_Y(-mE)$ is the extension of flat modules by induction. 
\end{proof}

\begin{prop}\label{p-localstable}
Let $\pi\colon (X,\Delta)\to B$ be a locally stable family over a smooth base $B$ with general fibers being  klt. Denote by $\sigma\colon B\to X$ a section. Let $\fa$ be an ideal on $X$ whose reduced cosupport is $\sigma(B)$. Assume $\mathcal{O}_X/\fa$ is flat over $B$, and for any $s\in B$,
$\lct(X_s,\Delta_s; \fa_s)=c$ is a constant. Assume $E$ is a prime divisor over $X$ which calculates the log canonical threshold, i.e. $c\cdot \ord_{E}(\fa)=A_{X,\Delta}(E)$. 

Then we have the following facts
\begin{enumerate}
\item  We can extract $E$ to obtain a model $\mu\colon Y\to X$ and $E$ dominates $B$; 
\item  For any $s\in B$, $\mu_s\colon Y_s\to X_s$ is birational (over each component of $X_s$) and $\pi_Y\colon (Y,E+\mu_*^{-1}\Delta)\to B$ is a locally stable family;
\item Write $\mu^{-1}\fa=\cO_Y(-\ord_E(\fa)E)\cdot \fb$ (see \eqref{e-Qideal}), then for any $s\in B$, $(Y_s,E_s+(\mu_s)^{-1}_*\Delta_s+\fb^c_s)$ is slc; and
\item $E$ is demi-normal so that we can define $K_E+\Delta_E:=(K_Y+E+\mu_*^{-1}\Delta)|_E$, and $(E,\Delta_E) \to B$ is locally stable. 
\end{enumerate}
\end{prop}
\begin{proof}Since $A_{X,\Delta}(E)=c\cdot \ord_E(\fa)$, it follows from \cite[1.3.1]{BCHM} that one can construct a model $\mu\colon Y\to X$ such that $E$ is the only exceptional divisor and $-E$ is ample over $X$.
Let $F$ be a divisor over $X$, and assume $F$ does not dominate $B$. Let $s$ be the generic point of its image on $B$ with ${\rm dim}(\cO_{s,B})=e$. Let $H_1,...,H_e$ be general hypersurfaces passing through $s$. Then since $(X,\Delta+\sum^e_{i=1}H_{i})$ satisfies that 
$$\lct(X,\Delta+\sum^e_{i=1}H_{i};\fa)=\lct(X_s,\Delta_s;\fa_s)=c,$$
by inversion of adjunction,
then 
$$A_{X,\Delta}(F)-c\cdot \ord_F(\fa)\ge \sum^e_{i=1} \ord_F(H_i)\ge 1>0.$$ Thus $E$ dominates $B$. This proves (1).

\medskip

To prove (2), we first show that $\pi_Y|_E$ is equi-dimensional. Otherwise, assume for some $s\in B$, we have $\dim (E_s)\ge \dim(X)-\dim (B)$. Denote by $e={\rm codim}_{E_s}Y$. Consider general hypersurfaces $H_1,..., H_e$ passing through $s$. Then $(X, \Delta+\pi^*\sum^e_{i=1}H_i+c\cdot \fa)$ is log canonical, and $A_{X, \Delta+\sum^e_{i=1}H_i}(E)=c\cdot \ord_E(\fa)$. Thus,
$(Y, E+\mu^{-1}_* \Delta+\pi_Y^*\sum^e_{i=1}H_i)$ is log canonical. However, the generic point of $E_s$ is of codimension $e$ but with $(e+1)$ $\mathbb{Q}$-Cartier Weil divisors $H_1$,..., $H_e$ and $E$ passing through it, which is a contradiction (see \cite[Prop. 34]{dFKX-dualcomplex}). This implies $Y\to B$ is equidimensional. Moreover, $Y_s\to X$ is birational, since if $Y_s$ has another component, then it is contained in $E$ and $\pi_Y^*\sum^e_{i=1}H_i$, thus violates \cite[Prop. 34]{dFKX-dualcomplex} again.
As $Y$ is potentially klt, $Y$ is CM, thus $Y\to B$ is flat.  In particular, $\mu_s\colon Y_s\to X_s$ is birational over each component of $X_s$. By Lemma \ref{l-localstable}, $(Y,E+\mu_*^{-1}\Delta)\to B$ is a locally stable family.

\medskip

We have shown $(X_s,\Delta_s)$, $(X_s,\Delta_s+\fa_s^c)$ and $(Y_s,E_s+(\mu_s)^{-1}_*\Delta_s)$ are slc. We also have shown that $\mu_s^{-1}(\fa_s)=\cO_{Y_s}(\sum_i a_i\Gamma_i)\cdot \fb_s$ with $a_i\cdot c=A_{X_s,\Delta_s}(\Gamma_i)$ for any exceptional divisor $\Gamma_i$ of $\mu_s$ by the above argument. Thus $(Y_s,E_s+(\mu_s)^{-1}_*\Delta_s+\fb^c_s)$ is slc. This proves (3).

\medskip

We proceed to prove (4). As $(Y,E+\mu^{-1}_*\Delta)$ is lc, $E$ is demi-normal by a classification of log canonical surface singularities. Since $E$ is $\mathbb{Q}$-Cartier, and $Y$ is potentially klt,  $E$ is Cohen-Macauley. In particular, $\pi_E\colon E\to B$ is flat since it is equidimensional. To prove it is locally stable, we can restrict over a general curve passing through $s$, and thus assume $B$ is a curve. Thus $(E^n, \Delta_{E^n})\to B$ is a locally stable family, since $E^n$ is flat over $B$
and $(K_{E^n}+E_s^n+\Delta_{E^n})=(K_Y+E+Y_s+\Delta)|_{E^n}$ which implies $({E^n},E_s^n+\Delta_{E^n})$ is lc. Thus $(E,\Delta_{E}+E_s)$ is slc which implies $(E,\Delta_E)\to B$ is locally stable over $s$.
\end{proof}

\begin{lem}\label{l-degslc}
Let $\pi\colon (X,\Delta)\to B$ be a locally stable family over a smooth base $B$ with general fibers being  klt. Denote by $\sigma\colon B\to X$ a section. Let $E$ be a prime divisor whose center on $X$ is $\sigma(B)$. Consider the proper morphism   $\mu\colon Y\to (X={\rm Spec}(R),\Delta)$ which yields the model with ${\rm Ex}(\mu)=E$ and $-E$ is ample,  and assume $(Y,E+\mu_*^{-1}\Delta)\to B$ is a locally stable family. 
Then $(\cX,\Delta_{\cX})\to \bA^1$ in Construction \ref{extendedrees}  yields a locally stable family over $B\times \bA^1$.
\end{lem}
\begin{proof}The calculation in \cite[Prop. 2.32]{LX-SDCII} implies that $K_\cX+\Delta_{\cX}$ is $\bQ$-Cartier. Since $\fa_m=\mu_*\cO_Y(-mE)$ satisfies $\mathcal{O}_X/\fa_m$ is flat over $B$ for any $m\in \mathbb Z_{\ge 0}$ by Lemma \ref{l-kvvanishing}, Construction \ref{extendedrees} commutes with a base change. Thus we know that for any $s\in B$,
$$(\cX,\Delta_{\cX})\times_{B\times \bA^1}(\{s\}\times \bA^1)\to \{s\}\times \bA^1$$
yields a degeneration of $(X_s,\Delta_s)$ to an orbifold cone over $(E_s,\Delta_{E_s})$ (see \cite[Sec. 2.4]{LX-SDCI}), which is slc by Proposition \ref{p-localstable}(4).
Thus $(\cX,\Delta_{\cX})\to \bA^1$ is a locally stable family over $B\times \bA^1$.
\end{proof}

Then we can state the following lemma, which is a key for our degeneration construction. 

\begin{lem}\label{l-equallct}Notation as in Proposition  \ref{p-localstable}. Let $(\cX,\Delta_{\cX})\to B\times \bA^1$ be the locally stable family constructed as in Lemma \ref{l-degslc}.
 Let $s\in B$ and 
 $$(\cX_s,\Delta_{\cX_s}):=(\cX,\Delta_{\cX})\times_{B\times \bA^1} (\{s\}\times \bA^1)$$ with the central fiber $(X_{s,0},\Delta_{{s,0}})$ over 0 the degeneration of $(X_s,\Delta_{s})$. Let $\bin(\fa_s)$ on $X_{s,0}$ be the degeneration of $\fa_s:=\fa \vert_{X_s}$. Then  $\lct\big(X_{s,0},\Delta_{{s,0}};\bin(\fa_s)\big)=c.$
\end{lem}
\begin{proof} By our assumption  $(Y,E+\mu_*^{-1}\Delta)\to B$ is a locally stable family.  Denote by $\mu^{-1}\fa=\mathcal{O}_Y(-aE)\cdot \fb$ as in \eqref{e-Qideal}. Then $c\cdot a= A_{X,\Delta}(E)$.

Restricting over $s$, we know that $(Y_s, E_s+(\mu_s)_*^{-1}\Delta_s+\fb_s^c)$ is slc by Proposition  \ref{p-localstable}(3). 
Moreover, 
$(E,\Delta_E)\to B$ is locally stable by Proposition  \ref{p-localstable}(4), so $(E_s,\Delta_{E_s})$ is slc. Applying adjunction, we have
$$(K_{Y_s}+ E_s+(\mu_s)_*^{-1}\Delta_s)|_{E_s}=K_{E_s}+\Delta_{E_s}\mbox{\ \ and \ \ } \fb_s|_{E_s}=:\fc_s.$$
It follows that $(E_s,\Delta_{E_s}+\fc_s^c)$ is slc. 

Since $(X_{s,0},\Delta_{s,0})$ is the degeneration of $(X_s,\Delta_s)$ under $(\cX_s,\Delta_{\cX_s})\to \bA^1$, $(X_{s,0},\Delta_{s,0})$ is an orbifold cone over $(E_s,\Delta_{E_s})$ (see \cite{LX-SDCI}*{Section 2.4}).
Moreover, if we blow up the vertex of the cone $\mu_{s,0}\colon Y_{s,0}\to X_{s,0}$, then the exceptional divisor is isomorphic to $E_s$.
Then by the construction, we know
\begin{eqnarray*}
\mu_{s,0}^*\big(K_{X_{s,0}}+\Delta_{s,0}+\bin(\fa)^c\big)|_{E_s}&=&(K_{Y_s}+ E_s+(\mu_s)_*^{-1}\Delta_s+\fb_s^c)|_{E_s}\\
&=&K_{E_s}+\Delta_{E_s}+\fc_s^c.
\end{eqnarray*} Therefore,  $\big(X_{s,0}, \Delta_{{s,0}}+\bin(\fa)^c\big)$ is slc by inversion of adjunction (see Proposition \ref{p-IOA}).
\end{proof}

\subsection{Equivariant degeneration}\label{ss-degeneration}

Repeatedly using Lemma \ref{l-equallct}, we have the following construction, which is a more precise version of  Theorem \ref{t-degeneration}.

\begin{thm}\label{t-equallc}
Let $x\in (X={\rm Spec}(R),\Delta)$ be a klt singularity. Let $\fa$ be an $\fm_x$-primary ideal on $X$. 
Assume there are prime divisors $E_1,...,E_r$ which calculate the log canonical threshold $c:=\lct_x(X,\Delta;\fa)$. 
 
 Then there is a $T(\cong\bG_{m}^r)$-equivariant locally stable family $\pi\colon (\cX_r,\Delta_{\cX_r})\to \mathbb{A}^r$ with a section $\sigma_r\colon \bA^{r}\to \cX_r$, such that $\pi^{-1}(s)\cong (X,\Delta)$ for any $s\in (\bA^1\setminus\{0\})^r$. 
 L

 Moreover, there is an ideal sheaf $\tilde{\fa}$ of $\mathcal{O}_{\cX_r}$ which extends $\fa\times T$ over $\mathbb A^r$,  such that the reduction of ${\rm Cosupp}(\tilde{\fa})$ is $\sigma(\bA^r)$ and for any $s\in \mathbb{A}^r$, $\lct(X_s,\Delta_s;\tilde{\fa}_s)=c$. 
 \end{thm}

\begin{proof}
When $r=0$, this is trivial. Assume we have proved the case $r=k$. Now we consider the case $r=k+1$.

From the induction, there is a locally stable family $\pi_k\colon (\cX_{k},\Delta_{\cX_k})\to \mathbb{A}^{k}$, with the ideal $\tilde{\fa}_k$ of $\cX_k$ which is flat over $\bA^{k}$, such that 
$$\tilde{\fa}_k|_{X\times (\bA^1\setminus\{0\})^{k}}=\fa\times (\bA^1\setminus\{0\})^{k},$$ 
with a section $\sigma_{k}\colon \bA^{k}\to \cX_k$. By the induction assumption, 
$$\lct\big((\cX_k)_s,(\Delta_{\cX_k})_s, (\tilde{\fa}_k)_s\big)=c$$ for any $s\in \bA^{k}$. In particular, $\lct(\cX_k,\Delta_{\cX_k}, \tilde{\fa}_k)=c$.

\medskip

 Let $\cE_{k+1}$ be the divisor which is birational transform of $E_{k+1}\times\bA^{k}$. Then $\cE_{k+1}$ computes the log canonical threshold $c=\lct(\cX_k,\Delta_{\cX_k},\tilde{\fa}_k)$, because
 $$A_{\cX_k,\Delta_{\cX_k}}(\cE_{k+1})=A_{X,\Delta}(E_{k+1})=c\cdot \ord_{E_{k+1}}(\fa)=c\cdot \ord_{\cE_{k+1}}(\tilde{\fa}_k).$$ 
 By Lemma \ref{l-extract}, we can extract $\cE_{k+1}$ to get a model $\mu_k\colon \cY_{k}\to \cX_{k}$, such that $-\cE_{k+1}$ is relatively ample.  Lemma \ref{l-localstable}(2) implies that over any $s\in \bA^{k}$, $(\cY_k)_s\to (\cX_k)_s$ is a birational morphism, thus 
$$(\cY_k,({\mu}_k)^{-1}_*\Delta_{\cX_k}+\cE_{k+1})\to \bA^{k}$$
is a locally stable family.

Define $\fb_i=(\mu_k)_*(-i\cE_{k+1})$ on $\cX_k$, then the reduction of ${\rm Cosupp}(\fb_i)$ is $\sigma_{k}(\bA^{k})$. It follows from $-\cE_{k+1}$ being relatively ample that $ \bigoplus_{i\in \mathbb{Z}} \fb_i$ is finitely generated,  and $\fb_i$ as well as $\fb_i/\fb_{i+1}$ are flat over $\bA^{k}$ by Lemma \ref{l-kvvanishing}. Thus we can apply  Construction \ref{extendedrees}  to get a flat family 
$$\cX_{k+1}:=\Spec (\bigoplus_{i\in \bZ} \fb_it^{-i})\to \bA^{k}\times \bA^{1}=\bA^{k+1},$$
with a section $\sigma_{k+1}\colon \bA^{k+1}\to \cX_{k+1}$. Moreover, write $\Delta_{\cX_k}=\sum_i d_i\Delta_{k,i}$. As in Definition \ref{d-degI}, we can define the degeneration $\Delta_{\cX_{k+1}}$ of $\Delta_{\cX_k}$ over $\bA^{k+1}$. Denote by $\tilde{\fa}_{k+1}$ the degeneration of $ \tilde{\fa}_k$.

\medskip

Lemma \ref{l-degslc} implies that $(\cX_{k+1},\Delta_{\cX_{k+1}})\to \bA^{k+1}$ is a locally stable family over $\bA^{k+1}$.

\medskip

Finally, to check the last claim for any $t\in \bA^{k+1}$, $\lct(\cX_t,(\Delta_{\cX})_t, (\tilde{{\fa}})_t)=c$.  We only need to check for a point $t$ of the form $(s,0)$ where $s\in \bA^{k}$. Since Construction \ref{extendedrees} commutes with base change, this can be considered as the degeneration induced by 
$$\Spec \Big(\bigoplus_{i\in \bZ}(\mu_k)_{s*}\mathcal{O}_{Y_s}(-i(\cE_{k+1})_s)t^{-i} \Big)\to \Spec k[t].$$ Since $(\mu_k)_s\colon Y_s\to X_s$ satisfies the assumption of Lemma \ref{l-equallct}, we can conclude.
\end{proof}

Let $\mu\colon Y\to X$ be a proper birational morphism from a normal birational model $Y$ such that $E_1,...,E_r\subset Y$. The family of $\cX_r\to \bA^r$ in Theorem \ref{t-equallc} can be explicitly given in the following way. 
\begin{prop}\label{p-totalring}
Notation as in Theorem \ref{t-equallc}. For any $1\le k \le r$, 
\begin{enumerate}[label=\alph*)]
\item the family $\cX_k\to \mathbb{A}^k$ is given by 
\begin{equation}\label{e-multirees}
\cX_k={\rm Spec} \bigoplus_{(i_1,...,i_k)\in \bZ^k} \mu_{*}\mathcal{O}_{Y}(-i_1E_1-\cdots -i_kE_k)~ t_1^{-i_1}\cdots t_k^{-i_k}
\end{equation}
\item for any $j> k$,  if we denote by $\widetilde{\mu_{k,j}}\colon  \cY_{k,j}\to \cX_k$ the morphism which extracts $\cE_j$ as in Theorem \ref{t-equallc}, then
$$\widetilde{\mu_{k,j}}_*  \cO_{\cY_{k,j}}(-m\cE_j)= \bigoplus_{(i_1,...,i_k)\in \bZ^k} \mu_*\mathcal{O}_{Y}(-i_1E_1-\cdots -i_kE_k-mE_j)~t_1^{-i_1}\cdots t_k^{-i_k}.$$
\end{enumerate}
\end{prop} 
\begin{proof}We will prove $b)_{\le k}+a)_{\le k}  \Rightarrow a)_{\le k+1}\Rightarrow b)_{\le k+1}$. Both $a)_0$ and $b)_0$ are trivial. 


Assume we have shown $a)_{k}$ which gives a family $\cX_k\to \bA^{k}$ and  $b)_{k}$. Let $\widetilde{\mu_{k}}\colon \cY_{k}\to \cX_{k}$ be the model extracting $\cE_{k+1}$ which is the birational transform of $E_{k+1}\times \bA^{k}$, 
then by $b)_k,$
$$\fb_i:=\widetilde{\mu_{k}}_* \cO_{\cY_{k}}(-i\cE_{k+1})= \bigoplus_{(i_1,...,i_k)\in \bZ^k} \mu_*\mathcal{O}_{Y}(-i_1E_1-\cdots -i_kE_k-iE_{k+1})~t_1^{-i_1}\cdots t_k^{-i_k}.$$
So by our construction, for the family $\cX:=\Spec \cR \to \bA^{k+1}$,
\begin{eqnarray*}
 \cR&=&\bigoplus_{i\in \bZ} \widetilde{\mu_{k}}_* \cO_{\cY_{k}}(-i\cE_{k+1})t^{-i}_{k+1}\\
  &=& \bigoplus_{i\in \bZ}   \bigoplus_{(i_1,...,i_k)\in \bZ^k} \mu_* \mathcal{O}_{Y}(-i_1E_1-\cdots -i_kE_k-iE_{k+1})~t_1^{-i_1}\cdots t_k^{-i_k} \cdot t^{-i}_{k+1}\\
&=& \bigoplus_{(i_1,...,i_k,i)\in \bZ^{k+1}}\mu_* \mathcal{O}_{Y}(-i_1E_1-\cdots -i_kE_k-iE_{k+1})~t_1^{-i_1}\cdots t_k^{-i_k} t^{-i}_{k+1},
\end{eqnarray*}
which says $a)_{k+1}$ holds. 

For $j>{k+1}$, let $\widetilde{\mu_{k+1,j}} \colon \cY_{k+1,j}\to \cX_{k+1}$, which extracts  the birational transform $\cE_{k+1,j}$ of $E_j\times \bA^{k+1}$ and let $\widetilde{\mu_{k,j}} \colon \cY_{k,j}\to \cX_{k}$ extract $\cE_{k,j}$ the birational transform of $E_j\times \bA^{k}$.

By $b)_k$, for any fixed $m$,
$$\widetilde{\mu_{k,j}}_* \cO_{\cY_{k,j}}(-m\cE_{k,j})= \bigoplus_{(i_1,...,i_k)\in \bZ^k} \mu_* \mathcal{O}_{Y}(-i_1E_1-\cdots -i_kE_k-mE_j)~t_1^{-i_1}\cdots t_k^{-i_k}. $$
Since $\widetilde{\mu_{k+1,j}}_* \cO_{\cY_{k+1,j}}(-m\cE_{k+1,j})$ is flat over $\mathbb{A}^{k+1}$, in particular that
$\widetilde{\mu_{k+1,j}}_* \cO_{\cY_{k+1,j}}(-m\cE_{k+1,j})$  is the flat degeneration family of  $\widetilde{\mu_{k,j}}_* \cO_{\cY_{k,j}}(-m\cE_{k,j})$ over $ \bA^1$, thus 
\begin{eqnarray*}
& &\widetilde{\mu_{k+1,j}}_* \cO_{\cY_{k+1,j}}(-m\cE_{k+1,j})	\\
& =&\bigoplus_{i_{k+1}\in \mathbb{Z}}\widetilde{\mu_{k,j}}_* \cO_{\cY_{k,j}}(-m\cE_{k,j})\cap \fb_{i_{k+1}} t_{k+1}^{-i_{k+1}} \hspace{5mm}\mbox{(by Lemma \ref{l-extendprimary} and  base-change)}\\
&=&\bigoplus_{i_{k+1}\in \mathbb{Z}} \left( \bigoplus_{(i_1,...,i_k)\in \bZ^k} \mu_* \mathcal{O}_{Y}(-i_1E_1-\cdots -i_kE_k-mE_j)~t_1^{-i_1}\cdots t_k^{-i_k} \right)\cap \fb_{i_{k+1}} t_{k+1}^{-i_{k+1}}\\
&=& \bigoplus_{i_{k+1}\in \mathbb{Z}} \left( \bigoplus_{(i_1,...,i_k)\in \bZ^k} \mu_* \mathcal{O}_{Y}(-i_1E_1-\cdots -i_kE_k-mE_j)\cap \fb_{i_{k+1}} \right)~t_1^{-i_1}\cdots t_k^{-i_k} t_{k+1}^{-i_{k+1}}\\
&=& \bigoplus_{(i_1,...,i_k,i_{k+1})\in \bZ^{k+1}} \mu_* \mathcal{O}_{Y}(-i_1E_1-\cdots -i_kE_k-i_{k+1}E_{k+1}-mE_j)~t_1^{-i_1}\cdots t_k^{-i_k} t_{k+1}^{-i_{k+1}},
\end{eqnarray*}
which is exactly the claim $b)_{k+1}$.
\end{proof}

\begin{cor}\label{c-deg}Let $s=(s_1,...,s_r)\in \bA^r$ and $s_{j_1},..., s_{j_k}$ are precisely the components which are $0$. For any $\ui=(i_1,...,i_k)\in \bZ^{k}$, let $\fb_{\ui}=\mu_* \mathcal{O}_{Y}(-i_1E_{j_1}-\cdots -i_{k}E_{j_k})$, then the fiber $X_s$ is isomorphic to $\Spec T$ where 
$T\cong \bigoplus_{\ui\in \bZ^k} \fb_{\ui}/\fb_{>\ui} .$

In particular, denote by
$R_*=\bigoplus_{\underline{i}\in \bZ^r_{\ge 0}} \fa_{\underline{i}}/\fa_{>\underline{i}},$
where $\fa_{\underline{i}}=\mu_* \mathcal{O}_{Y}(-i_1E_1-\cdots -i_rE_r)$. 
Then $X_0$ is isomorphic to $\Spec R_*$.
\end{cor}

Write $e_k$ be the $k$-th $\bG_m$ in $\bG_m^r$. Let $w:=(w_{1},...,w_r)\in \bN^r_{+}$, and the coweight  $\xi_w:=\sum^r_{k=1} w_ke_k$. 
The morphism $i_w\colon \bG_m\to (\bG_m)^r$, can be extended to a morphism $\bar{i}_w\colon\bA^1\to \bA^r$ such that $\bar{i}_w(0)=0$. Thus the base change family will have its special fiber isomorphic to $X_0:=\cX\times_{\bA^r}\{0\}$.

\bigskip

Now we assume the degeneration $X_0$ is irreducible. This implies that the filtration induced by $\xi_w$, which satisfies
$$\wt_{\xi_w}(f):=\wt_{\xi_w}(\bar{f})=\sum^r_{k=1}w_k\cdot\ord_{E_k}(f),$$
is a valuation by Lemma \ref{l-filtrationrestriction}.

Since this is true for any $f\in R$, we know that for any element $\frac{f}{g}\in K$ with $f,g\in R$, 
\begin{equation}\label{e-degeneration}
\wt_{\xi_w}(\frac{f}{g})=\wt_{\xi_w}(f)-\wt_{\xi_w}(g)=\sum^r_{k=1}w_k\cdot(\ord_{E_k}(f)-\ord_{E_k}(g))=\sum^r_{k=1}w_k\cdot\ord_{E_k}(\frac{f}{g}).
\end{equation}
Thus this implies the following statement. 
\begin{lem}\label{l-d=w}
With the notation of Theorem \ref{t-equallc}, we assume $X_0$ is irreducible. Assume  $\mu\colon Y\to (X,\Delta)$ is a birational model such that $ E_1,..., E_r$ are contained in the exceptional divisor and  $Z$ a component of their intersection with $(\eta(Z)\in Y, E_1+\cdots+E_r)$ being qdlt, then $\wt_{\xi_w}$ is the same as the quasi-monomial valuation $v_{w}$ of at $(\eta(Z)\in Y, E_1+\cdots+E_r)$ with weight $(w_1,...,w_r)$. (In particular, such $Z$ is unique.)
\end{lem}
\begin{proof}Let $z_1,...,z_r$ be functions in $\cO_{Y,\eta(Z)}$ such that $z_i$ is the equation of $E_i$ near $\eta(Z)$. For any $f$ in $\cO_{Y,\eta(Z)}$, we can write 
$$f=\sum_{\beta\in \bZ^{r}_{\ge 0}}c_{\beta}z^{\beta}\mbox{ where $c_{\beta}\neq 0$}. $$
By definition $v_w(f)=\min_{\beta} \langle \beta, w\rangle$. By \eqref{e-degeneration}, we know $\wt_{\xi_{w}}(f)\le v_w(f)$. On the other hand, since $\wt_{\xi_w}$ is a valuation, we have 
$$\wt_{\xi_{w}}(f)\ge \min_{c_{\beta}\neq 0}\wt_{\xi_{w}}(z^{\beta})=\min \langle \beta, w\rangle=v_w(f).$$
\end{proof}

\begin{proof}[Proof of Theorem \ref{t-main2}]First assume the existence of $E_1,...,E_r$.
We may assume the coordinates $w_i$ of $w$ are all positive. It follows from Lemma \ref{l-d=w}, that if $X_0$ is irreducible, the associated graded ring for $v_{w}$ is the same as the flitration induced by $\xi$, which is $R_*$.

\medskip

For the converse, the cone $\LC(X,\Delta+\fa^c)$  consisting of all lc places of $(X,\Delta+\fa^c)$ contains a minimal rational cone of dimension $r$ whose interior contains $v$. It follows from \cite[Lemma 2.10]{LX-SDCII} that we can pick up $E_1,...,E_r$ contained in this cone which satisfy Condition (1) and (2).
\end{proof}

\subsection{Finite generation implies optimal degeneration}\label{s-global}

Let $(X,\Delta)$ be a log Fano pair, i.e. $(X,\Delta)$ is klt and $-K_X-\Delta$ is ample. Let  $r\in \bN_{+}$ such that $-r(K_X+\Delta)$ is Cartier. We want to consider the finite generation for the graded rings 
$$R:=\bigoplus_mR_m=\bigoplus_m H^0(-rm(K_{X}+\Delta)).$$ 

In the last section, we will show Conjecture \ref{c-main} implies the existence of an optimal degeneration. The implication is probably well known among experts. 

\begin{prop}\label{p-optimal} 
For a log Fano pair $(X,\Delta)$, assume $\delta(X,\Delta)\le 1$ and there is a quasi-monomial valuation $v$ computing $\delta(X,\Delta)$ such that $\gr_v(R)$ is finitely generated.  Then there exists a divisor $E$ over $X$ computing $\delta(X,\Delta)$. In particular, $(X,\Delta)$ is not K-stable. 
\end{prop}

Recall in \cite{BX-uniqueness} and \cite{BLZ-twisted}, it is shown that the existence of a divisor $E$ that computes $\delta(X,\Delta)\le 1$ is equivalent to the existence of {\it an optimal degeneration}, that is a special degeneration of $(X,\Delta)$, such that the central fiber $(X_0,\Delta_0)$ satisfies $\delta(X_0,\Delta_0)=\delta({X,\Delta})$ which is computed by the divisor induced by the $\bG_m$-action on $(X_0,\Delta_0)$.

\begin{proof}By our assumption, we know ${\rm gr}_vR$ is finitely generated. From \cite{LX-SDCII}*{Lemma 2.10} (see Theorem \ref{t-main2}), we know that there exists a $ $ model $Y\to (X,\Delta+D)$, and prime divisors $E_1,..., E_r$ on $Y$ with an intersection $Z$ such that $(\eta(Z)\in Y,E_1+\cdots+E_r)$ yields a toroidal structure, and $v=v_{w_0}$ for some $w_0\in \mathbb{R}^r_{>0}$. Moreover, $E_1$,..., $E_r$ induce a torus degeneration of $R$ over $\mathbb{A}^r$ where $r={\rm rat. rk.}(v)$, such that if we denote by $R_*$ the central fiber of the degeneration, then $R_*={\rm gr}_vR$. 

Moreover, there is a $\bG_m\times \bG_m^r$ actions on $R_*$ that yields a grading $R_*=\bigoplus_{m\in \bZ_{\ge 0}} (R_*)_m$, where the action of the first factor is induced by the grading of $R$ and we can write $(R_*)_m=\bigoplus_{\ui\in \bZ^r_{\ge 0}} (R_*)_{m,\ui}$, where $\ui$ corresponds to a weight of $\bG^r_m$. 
By Lemma \ref{l-d=w}, for any $w \in \bR^r_{\ge 0}$ and $f\in R_{m}$, 
$v_{w}(f)= \langle w , \ui\rangle$ where $\ui$ is the unique $r$-tuple in $\bZ^r$ whose $k$-th coordinate is given by $\ord_{E_k}(f)$.

\begin{claim} The function $S_{X,\Delta}(v_{w})$ $(w\in  \bR^r_{\ge 0})$ is a linear function on the cone $ \bR^r_{\ge 0}$. 
\end{claim}
\begin{proof}For a fixed $m\in \bN$, define
$H_{w,m}=\sum_t t\cdot F_{v_{w}}^t(R_m)/F_{v_{w}}^{>t}(R_m),$
and by definition we have
$$S_{X,\Delta}(v_{w})=\lim_{m\to \infty}\frac{H_{w,m}}{mr\cdot \dim(R_m)}.$$

Since $v_{w}(f)= \langle w , \ui\rangle$, we have 
$H_{w,m}=\sum_{\ui}\dim \langle {w},\ui \rangle \cdot \dim(R_{\ui}),$
thus $H_{w,m}$ is linear on $\mathbb{R}^r_{\ge 0}$, which then implies $S_{X,\Delta}(v_{w})$ is linear on $\mathbb{R}^r_{\ge 0}$.
\end{proof}
Returning to the proof, since $\delta(v_{w})=\frac{S_{X,\Delta}(v_{w})}{A_{X,\Delta}(v_{w})}$, and both the numerator and denominator are linear with respect to $w$, we know if $\delta(v_{w})$ attains the minimium for some $w_0\in \bR^r_{>0}$, then the function $\frac{A_{X,\Delta}(v_{w})}{S_{X,\Delta}(v_{w})}$ is constant for any $w\in U$.
In particular, for any divisorial valuation $\ord_E$ corresponding to an element $w$ in $\bN^r$, it computes the minimum of $\delta(X,\Delta)$.  
\end{proof}



\end{document}